\documentclass{amsproc}

\usepackage[ansinew]{inputenc}
\usepackage{graphicx}
\usepackage{color}

\usepackage{enumerate,latexsym}
\usepackage{latexsym}
\usepackage{amsmath,amssymb}
\usepackage{graphicx}
\usepackage{times}

\newfont{\bb}{msbm10 at 11pt}
\newfont{\bbsmall}{msbm8 at 8pt}

\def\rth{\mathbb{R}^3}
\def\R{\mathbb{R}}

\def\N{\mathbb{N}}

\def\Hip{\mathbb{H}}

\def\esf{\mathbb{S}}

\newcommand{\la}{\looparrowright}

\newcommand{\ben}{\begin{enumerate}}
\newcommand{\bit}{\begin{itemize}}
\newcommand{\een}{\end{enumerate}}
\newcommand{\eit}{\end{itemize}}
\newcommand{\wh}{\widehat}
\newcommand{\Int}{\mbox{\rm Int}}

\newcommand{\wt}{\widetilde}

\newcommand{\sol}{\mathrm{Sol}_3}
\newcommand{\ed}{\end{document}}

\def\a{{\alpha}}

\def\t{{\theta}}

\def\G{{\Gamma}}

\def\L{\Lambda}

\def\be{{\beta}}

\def\cB{\mathcal{B}}

\def\cM{\mathcal{M}}

\let\8=\infty \let\0=\emptyset

\def\flecha{\rightarrow}
\def\esiz{\langle}
\def\esde{\rangle}

\def\cte.{\mathop{\rm cte.}\nolimits}

\def\N{\mathbb{N}}

\def\R{\mathbb{R}}

\def\H{\mathbb{H}}
\def\S{\mathbb{S}}
\def\SS{\Sigma}

\newtheorem{theorem}{Theorem}[section]
\newtheorem{lemma}[theorem]{Lemma}
\newtheorem{proposition}[theorem]{Proposition}

\newtheorem{remark}[theorem]{Remark}
\newtheorem{corollary}[theorem]{Corollary}
\newtheorem{definition}[theorem]{Definition}

\newtheorem{example}[theorem]{Example}

\newcommand{\su}{{\rm SU}(2)}
\newcommand{\EE}{\wt{\mathrm E}(2)}

\renewcommand{\sl}{\wt{\mathrm{SL}}(2,\R)}

\numberwithin{equation}{section}

\begin{document}

\begin{title}
{Embeddedness of spheres in homogeneous three-manifolds}
\end{title}
\today

\author{William H. Meeks III}
\address{William H. Meeks III, Mathematics Department,
University of Massachusetts, Amherst, MA 01003}
\email{profmeeks@gmail.com}
\thanks{The first author was supported in part by NSF Grant DMS -
   1004003. Any opinions, findings, and conclusions or recommendations
   expressed in this publication are those of the authors and do not
   necessarily reflect the views of the NSF}
\author{Pablo Mira}
\address{Pablo Mira, Department of
Applied Mathematics and Statistics, Universidad Polit\'ecnica de
Cartagena, E-30203 Cartagena, Murcia, Spain.}
\email{pablo.mira@upct.es}
\thanks{The second author was partially supported by the
MICINN-FEDER grant no. MTM2013-43970-P, and Programa de Apoyo
a la Investigacion, Fundacion Seneca-Agencia de Ciencia y
Tecnologia Region de Murcia, reference 19461/PI/14}

\author{Joaqu\'\i n P\'erez}
\address{Joaqu\'\i n P\'erez, Department of Geometry and Topology,
University of Granada, 18001 Granada, Spain}
 \email{jperez@ugr.es}

\thanks{The third author was partially
supported by the MINECO/FEDER grant no. MTM2014-52368-P}




\begin{abstract}
Let $X$ denote a metric Lie group diffeomorphic
to $\R^3$ that admits an algebraic open book decomposition.
In this paper we prove that if $\Sigma$ is an immersed surface
in $X$ whose left invariant Gauss map is a diffeomorphism onto
$\S^2$, then $\Sigma$ is an embedded sphere. As a consequence,
we deduce that any constant mean curvature sphere of index one
in $X$ is embedded.
\end{abstract}

\maketitle

\section{Introduction.}  \label{sec:introduction}

Let $Y$ denote a simply connected, homogeneous Riemannian
three-manifold, and assume that $Y$ is not isometric to a
Riemannian product  $\S^2(\kappa)\times \R$, where $\S^2(\kappa)$ is the two-sphere
with a metric of constant Gaussian curvature $\kappa>0$. Then,
$Y$ is isometric to a \emph{metric Lie group} $X=(G,\esiz,\esde)$, i.e.,
a simply connected, three-dimensional Lie group
$G$ equipped with a left invariant metric $\esiz,\esde$.

In this paper we consider immersed oriented surfaces in metric Lie groups, and we
study them in terms of their \emph{left invariant Gauss map}, which we define next.

Let $\psi\colon \Sigma\looparrowright  X$ be an immersed oriented surface in a metric
Lie group $X$, and let $N\colon \Sigma\flecha TX$ denote its unit normal vector field. Note
that for any $x\in X$, the left translation $l_x \colon X\flecha X$ is an
isometry of $X$. Thus, for every $p\in \Sigma$ there exists a unique
unit vector $G(p)$ in the tangent space  $T_e X$, such that
\[
\left( dl_{\psi(p)}\right) _e (G(p)) = N(p), \hspace{1cm }\forall p\in \Sigma,
\]
where $e$ denotes the identity element of $X$.

\begin{definition}\label{defG}
We call the map $G\colon \Sigma\flecha \S^2=\{v\in T_e X \, \mid \, |v|=1\}$
the \emph{left invariant Gauss map} of the
oriented surface $\psi\colon \Sigma\looparrowright  X$.
\end{definition}

Note that if $X$ is the Euclidean space $\R^3$ endowed with its
usual abelian Lie group structure, the left invariant Gauss map
is just the usual Gauss map for oriented surfaces in $\R^3$. In this situation,
it is well known that if $\Sigma$ is a surface in $\R^3$ whose Gauss map
is a diffeomorphism onto $\S^2$, then $\Sigma$ bounds a strictly convex domain of $\R^3$;
in particular, $\Sigma$ is an embedded topological sphere.

This embeddedness property does not hold in the general context
of metric Lie groups. For instance, for certain metric Lie groups $X$
diffeomorphic to $\S^3$, there exist immersed, non-embedded spheres in
$X$ whose left invariant Gauss maps are diffeomorphisms onto $\S^2$
(see Remark~\ref{torralbo}). It is then natural to investigate for which
metric Lie groups $X$ it is true that any sphere in $X$ whose left invariant
Gauss map is a diffeomorphism must be embedded.

In this paper we give an affirmative answer to the embeddedness question above for a
certain class of metric Lie groups, namely, for those admitting \emph{algebraic open
book decompositions} (see Definition~\ref{def:al-book}).

\begin{theorem}
 \label{main}
Let $X$ be a metric Lie group which admits an algebraic open book
decomposition, and let $f\colon  S\looparrowright X$ be an immersion of a
sphere whose
left invariant Gauss map $G\colon  S\to \esf^2$
is a diffeomorphism. Then $f(S)$ is an embedded sphere (i.e., $f$ is an
injective immersion).
\end{theorem}

We point out that the metric Lie groups which admit an algebraic open
book decomposition were classified in the lecture notes~\cite{mpe11} by the first and third authors;
see Section~\ref{sec:background} in this paper for more details. This class includes
the hyperbolic three-space $\H^3$, the Riemannian
product $\H^2\times \R$, the Riemannian Heisenberg space ${\rm Nil}_3$ and the solvable Lie
group $\sol$ with any of its left invariant metrics.
We note that Theorem~\ref{main} was previously known
only in the particular case $X=\R^3$.

Theorem~\ref{main} has an interesting application to the context of constant
mean curvature surfaces. Indeed, if $\Sigma$ is an immersed oriented sphere of constant
mean curvature in a metric Lie group $X$, and if $\Sigma$ has index one for its
stability operator, it was proved in~\cite{mmpr4} that the left invariant
Gauss map of $\Sigma$ is a diffeomorphism (this property was previously
proved by Daniel and Mira in~\cite{dm2} for the case where $X$ is the
three-dimensional Thurston geometry for $\sol$). Thus, the next corollary
follows immediately from Theorem~\ref{main}:

\begin{corollary}\label{maincor}
Let $X$ be a metric Lie group which admits an algebraic open book
decomposition, and let $\Sigma$ be an immersed constant mean curvature
sphere of index one in $X$. Then $\Sigma$ is embedded.
\end{corollary}

Corollary~\ref{maincor} was previously known in the following cases:
 \begin{enumerate}
 \item
If $X$ has constant curvature, then all constant mean curvature spheres
in $X$ are totally umbilical round spheres (by Hopf's classical theorem);
in particular they are embedded and have index one.
 \item
If $X$ has isometry group of dimension $4$, Abresch and Rosenberg~\cite{abro1,abro2} proved
that constant mean curvature
spheres in $X$ are rotational spheres. If, additionally, $X$ is
diffeomorphic to $\R^3$, then all such spheres have
index one~\cite{so3,tou1} and are embedded. We note that all homogeneous manifolds
diffeomorphic to $\R^3$ that have a $4$-dimensional isometry group
also admit a metric Lie group structure with an algebraic open book
decomposition (see Remark~\ref{openek}).
 \item
 If $X$ is the Lie group $\sol$ endowed with its \emph{standard} maximally
 symmetric left invariant metric, then the statement in Corollary~\ref{maincor}
 was proved by Daniel and Mira in~\cite{dm2}. We note that our proof here
 is completely different from the approach in~\cite{dm2}. As a matter of
 fact, the proof in~\cite{dm2} uses that the {standard} left invariant
 metric in $\sol$ admits planes of reflectional symmetry, a property that
 is not true for arbitrary left invariant metrics on the Lie group $\sol$.
 \end{enumerate}

\begin{remark}
\emph{In view of Theorem 4.1 in~\cite{mmpr4}, Corollary~\ref{maincor}
represents an advance towards proving the following conjecture:
\emph{any constant mean curvature sphere in a homogeneous manifold diffeomorphic to $\R^3$
is embedded}}.
\end{remark}

In Section~\ref{sec:background} we review several aspects of the geometry
of metric Lie groups that we will need. In particular, we will explain the
classification of metric Lie groups that admit an algebraic open book
decomposition. Section~\ref{sec:background} can be seen as introductory
material; for a more complete introduction to the geometry of surfaces in metric Lie
groups, we refer to~\cite{mpe11} and \cite{mil2}.

In Section~\ref{sec:openbooks} we prove a more detailed version of
Theorem~\ref{main}, which gives relevant information (beyond embeddedness)
about the geometry of spheres whose left invariant  Gauss maps are
diffeomorphisms (see Theorem~\ref{thm3.4}).

\section{Metric Lie groups and algebraic open book decompositions.}
 \label{sec:background}

We next consider metric Lie groups that admit an \emph{algebraic
open book decomposition}. This notion is an extension to metric Lie groups
of the usual notion of a pencil of half-planes in $\R^3$.

\begin{definition} \label{def:al-book}
{\rm Let $X$ be a metric Lie group and $\G \subset X $ a 1-parameter
subgroup. An {\it algebraic open book decomposition of $X$ with
binding $\G$} is a foliation $\mathcal{B}=\{ L(\theta )\} _{\t \in
[0,2\pi )}$ of $X-\G $ such that the sets
\[
H(\t )=L(\t )\cup \G \cup L(\pi +\t )
\]
are two-dimensional subgroups of $X$, for all $\t \in [0,\pi )$. We
will call $L(\t )$ the {\it leaves}  and $H(\t )$ the {\it
subgroups} of the algebraic open book decomposition $\mathcal{B}$.}
\end{definition}

Observe that this definition only depends on the Lie group
structure of $X$, and not on its left invariant metric.

The metric Lie groups that admit an algebraic open book decomposition
were classified by the first and third authors in~\cite{mpe11}. All of
them are \emph{semidirect products}. We will next review some properties
of metric semidirect products in order to present this classification.

\subsection{Semidirect products.}
Consider a \emph{semidirect product} which is a Lie group
$(\R^3\equiv \R^2\times \R,*)$, where the group operation $*$ is
expressed in terms of some real $2\times2$ matrix $A\in \cM_2(\R)$ as
\begin{equation}
\label{eq:5}
 ({\bf p}_1,z_1)*({\bf p}_2,z_2)=({\bf p}_1+ e^{z_1 A}\  {\bf
 p}_2,z_1+z_2);
\end{equation}
here $e^B=\sum _{k=0}^{\infty }\frac{1}{k!}B^k$ denotes the usual exponentiation
of a matrix $B\in \cM_2(\R )$. We will use the
notation $\R^2\rtimes_A \R$ to denote such a Lie group.

Suppose $X$ is isomorphic to $\R^2\rtimes_A
\R$, where
\begin{equation} \label{equationgenA} A=\left(
\begin{array}{cr}
a & b \\
c & d \end{array}\right) .
\end{equation}
Then, in terms of the coordinates $(x,y)\in \R^2$, $z\in \R $, we
have the following basis $\{ F_1,F_2,F_3\} $ of the space of {\it right invariant}
vector fields on $X$:
\begin{equation}
\label{eq:6}
 F_1=\partial _x,\quad F_2=\partial _y,\quad F_3(x,y,z)=
(ax+by)\partial _x+(cx+dy)\partial _y+\partial _z.
\end{equation}
In the same way, a {\it left invariant}  frame $\{ E_1,E_2,E_3\} $ of $X$
is given by
\begin{equation}
\label{eq:6*}
 E_1(x,y,z)=a_{11}(z)\partial _x+a_{21}(z)\partial _y,\quad
E_2(x,y,z)=a_{12}(z)\partial _x+a_{22}(z)\partial _y,\quad
 E_3=\partial _z,
\end{equation}
where
\begin{equation}
\label{eq:exp(zA)}
 e^{zA}=\left(
\begin{array}{cr}
a_{11}(z) & a_{12}(z) \\
a_{21}(z) & a_{22}(z)
\end{array}\right) .
\end{equation}
In terms of $A$, the Lie bracket relations are:
\begin{equation}
\label{eq:8a} [E_1,E_2]=0,
\quad
[E_3,E_1]=aE_1+cE_2,
\quad
 [E_3,E_2]=bE_1+dE_2.
 \end{equation}
Observe that ${\rm Span}\{ E_1,E_2\}$ is an integrable
two-dimensional distribution of $X$, whose integral surfaces define the
foliation $\mathcal{F}= \{ \R^2\rtimes _A\{ z\} \mid z\in \R \} $ of
$\R^2\rtimes _A\R$. All the leaves of $\mathcal{F}$ are complete and intrinsically flat
with respect to the metric given in the following definition.

\begin{definition}
 \label{def2.1}
 {\rm
We define the {\it canonical left invariant metric} on the
semidirect product $\R^2\rtimes _A\R $ to be that one for which the
left invariant frame $\{ E_1,E_2,E_3\} $ given by (\ref{eq:6*}) is
orthonormal. Equivalently, it is the left invariant extension to $X=\R^2\rtimes _A\R $
of the inner product on the tangent space $T_eX$ at the identity element $e=(0,0,0)$
that makes $(\partial _x)_e,(\partial _y)_e,(\partial _z)_e$ an orthonormal basis.}
\end{definition}

\subsection{Unimodular semidirect products.}
\label{secunimodgr}
A semidirect product $\R^2 \rtimes_A\R$ is unimodular if and only
if ${\rm trace}(A)=0$. We next provide four examples of unimodular
metric semidirect products. It was proved in~\cite{mpe11} that if $X$
is a unimodular metric Lie group that is not isomorphic to $\su$ or to
$\sl$, then $X$ is isomorphic and (up to rescaling) isometric to one of these four examples.

\begin{example}
\label{example1}
\emph{{\bf The abelian group $\R^3$}. If the matrix $A$ is zero, the metric
Lie group $X=\R^2\rtimes_A \R$ is the abelian group $\R^3$ endowed with
a left invariant metric. All such left invariant metrics are isometric
to the usual Euclidean metric of $\R^3$.}
\end{example}

\begin{example}
\label{example2}
\emph{{\bf The Heisenberg group}.
If $A=\left(\begin{array}{ll} 0 & 1 \\ 0 & 0\end{array}\right)$,
then $X=\R^2\rtimes_A\R$ gives a metric Lie group isomorphic
to the Heisenberg group ${\rm Nil}_3$. Up to
rescaling, ${\rm Nil}_3$ with an arbitrary left invariant metric
is  isometric to $X$. We note that all left
invariant metrics on ${\rm Nil}_3$ have a four-dimensional
isometry group.}
\end{example}

\begin{example}
\label{example3}
\emph{{\bf The group $\sol$}.
If $A=\left(\begin{array}{cl} 0 & c \\ 1/c & 0\end{array}\right)$,
$c\geq 1$, then $X_c=\R^2\rtimes_A\R$ gives a metric Lie group isomorphic
to the Lie group  ${\rm Sol}_3$. Specifically, for each $c\geq 1$ the
metric Lie group $X_c$ (with its canonical metric) is isomorphic
 to $\sol$. Conversely,
up to rescaling, all left invariant metrics in ${\rm Sol}_3$ are recovered by
this model. We note that $c=1$ corresponds to the Thurston geometry model
of $\sol$, i.e., to the left invariant metric in $\sol$ with a maximal isometry
group. Every left invariant metric of $\sol$ has an isometry group of dimension three.}
\end{example}

\begin{example}
\emph{{\bf The group $\EE$}.
Suppose $A=\left(\begin{array}{cr} 0 & -c \\ 1/c & 0\end{array}\right)$,
$c\geq 1$. Then $X_c=\R^2\rtimes_A\R$ gives a metric Lie group isomorphic
to $\EE$, the Lie group defined as the universal cover of the Euclidean
group of orientation-preserving rigid motions of the plane. These
metric Lie groups do not admit algebraic open book decompositions, so we will
not be more specific about them here.}
\end{example}

\subsection{Non-unimodular groups.}
The cases $X=\R^2\rtimes_A \R$ with ${\rm trace} (A)\neq 0$
correspond to the simply connected, three-dimensional, {\em non-unimodular}
metric Lie groups. In these cases,
up to the rescaling of the metric of $X$, we may assume that
${\rm trace} (A)=2$. {\it This normalization in the non-unimodular case
will be assumed from now on throughout the paper}.
After an orthogonal change of the left invariant frame (see Section~2.5 in~\cite{mpe11} for details),
we may express the matrix $A$ uniquely as
\begin{equation}
\label{Axieta} A=A(a,b)= \left(
\begin{array}{cc}
1+a & -(1-a)b\\
(1+a)b & 1-a \end{array}\right), \hspace{1cm} a,b\in [0,\infty).
\end{equation}

The \emph{canonical basis} of the
non-unimodular metric Lie group $X$ is, by definition, the
left invariant orthonormal frame $\{E_1,E_2,E_3\}$ given in
\eqref{eq:6*} by the matrix $A$ in~(\ref{Axieta}).
In other words, every non-unimodular metric Lie
group is isomorphic and isometric (up to possibly rescaling the
metric) to $\R^2\rtimes _A\R $ with its canonical metric, where $A$
is given by (\ref{Axieta}). If  $A=I_2$ where $I_2$ is the identity
matrix, then  we get a metric Lie group that we denote by $\H^3$, which is
isometric to the  hyperbolic three-space with its standard
constant curvature $-1$ metric and  where the underlying Lie group structure is
isomorphic to that of
the set of similarities of $\R^2$.
Under the assumption that $A\neq I_2$, the
determinant of $A$ determines uniquely the Lie group structure. This number
is the \emph{Milnor $D$-invariant} of $X=\R^2\rtimes_A \R$:
\begin{equation}
\label{Dxieta} D=(1-a^2)(1+b^2) ={\rm det} (A).
\end{equation}

Assuming $A\neq I_2$, given $D\in \R $, one can solve
(\ref{Dxieta}) for $a=a(D,b)$, producing a related matrix $A(D,b)$
by equation (\ref{Axieta}), and the space of canonical left
invariant metrics on the corresponding non-unimodular Lie group
structure is parameterized by the values of $b\in [m(D),\infty )$,
where
\begin{equation}
\label{eq:m(D)}
m(D)=\left\{ \begin{array}{cl}
\sqrt{D-1} & \mbox{ if $D>1$,}\\
0 & \mbox{ otherwise}.
\end{array}\right.
\end{equation}
In particular, the space of simply connected, three-dimensional,
non-unimodular metric Lie groups with a given $D$-invariant
is two-dimensional (one-dimensional after identification by rescaling of the metric).
See Figure~\ref{nonunimod} for a representation of these metric Lie groups in terms of
$(D,b)$.

\begin{figure}
\begin{center}
\includegraphics[height=5.3cm]{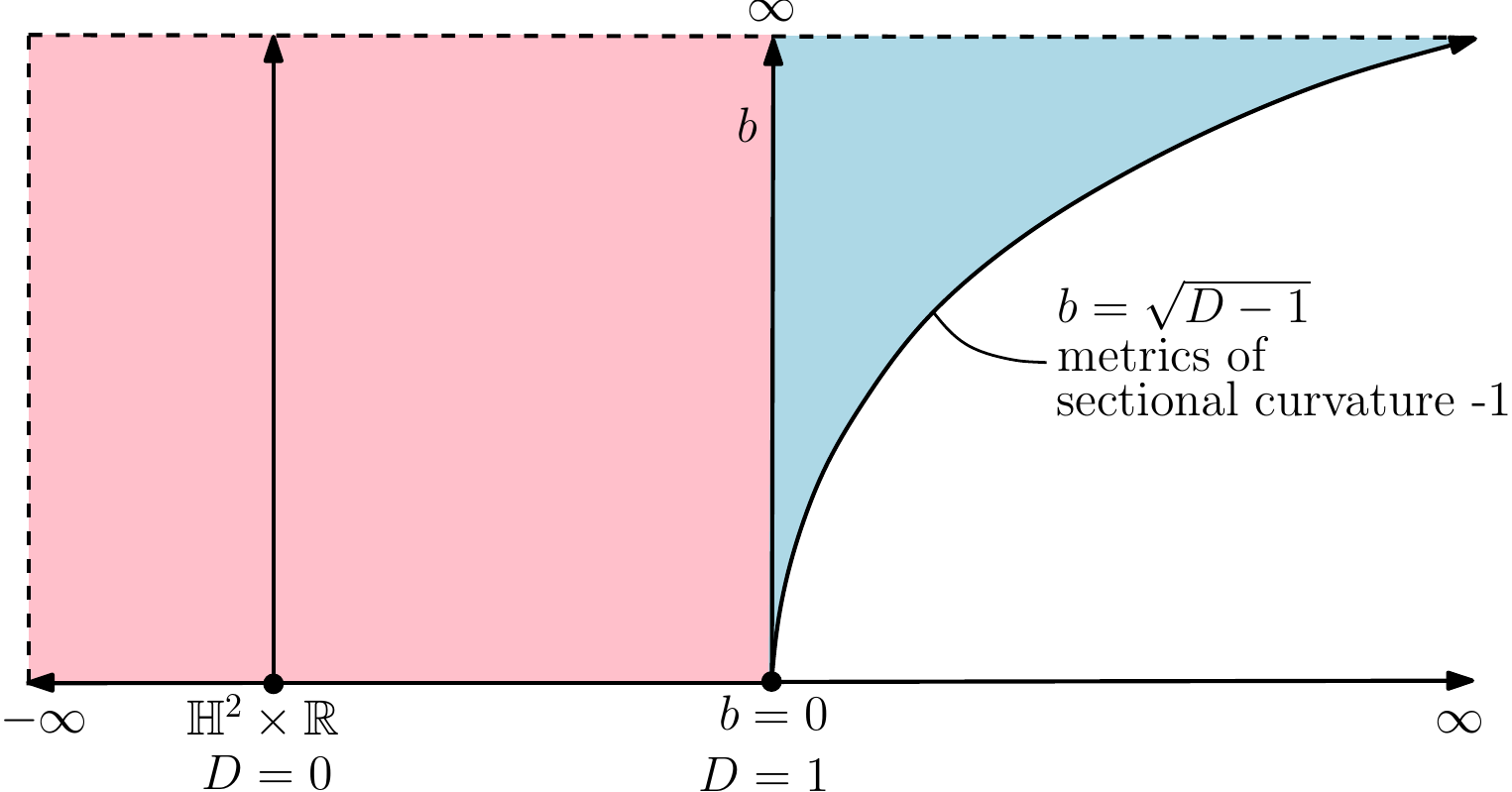}
\caption{Representation of the moduli space of
simply connected, three-dimensional, non-unimodular
metric Lie groups in terms of points in the $(D,b)$-plane, so that the group is
$\R^2\rtimes _{A(D,b)}\R $ where $A(D,b)\in \mathcal{M}_2(\R )$
is given by (\ref{Axieta}) after solving
(\ref{Dxieta}) for $a=a(D,b)$. Points in the same vertical halfline correspond to metric Lie groups
with the same Lie group structure.
The point $(D,b)=(0,0)$
corresponds to the product homogeneous manifold $\Hip ^2 (-4)\times
\R $ where $\Hip ^2 (-4)$ has constant negative curvature $-4$.
} \label{nonunimod}
\end{center}
\end{figure}

\subsection{Classification of algebraic open book decompositions}
\label{subsec:openbook}

The following result was proved by the first
and third authors, see Theorem~3.6 in~\cite{mpe11}.

\begin{proposition}\label{clasopenbook}
Let $X$ be a metric Lie group. The following properties are equivalent:
\begin{enumerate}
\item
$X$ admits an algebraic open book decomposition.
 \item
$X$ is isomorphic to one of the following Lie groups:
 $\R^3$, ${\rm Nil}_3$, ${\rm Sol}_3$ or a non-unimodular
 Lie group with $D$-invariant $D\leq 1$.
\end{enumerate}
\end{proposition}

As a matter of fact, the results in Section~3 of~\cite{mpe11} also
classify the algebraic open book decompositions and prove that when $X$
admits such an algebraic open book decomposition, then every two-dimensional
subgroup of $X$ is a subgroup of some open book decomposition of
$X$.

It follows from this classification that if
$X$ admits an algebraic open book decomposition $\mathcal{B}$, then
$X$ is isomorphic to $\R^2\rtimes_A \R$ for some $A\in
\mathcal{M}_2(\R)$, and we may assume that the binding $\G$ associated to $\mathcal{B}$
is the 1-parameter
subgroup\footnote{Note that exchanging  the matrix $A$ in
Example~3.5 in~\cite{mpe11} by its transpose, we do not change the
Lie group structure and the binding $\G = \{ (0,y,0)\ | \ y\in \R \} $ in that
example changes to $\G '=\{ (x,0,0)\ | \ x\in \R \} $.}
 $\{(x,0,0) \mid x\in \R\}$ and the $(x,y)$-plane $P_0=\R^2 \rtimes _A \{0\}$
is one of the subgroups of $\mathcal{B}$.

\begin{remark}\label{openek}
\emph{Let $Y$ be a homogeneous manifold diffeomorphic to $\R^3$,
and assume that the isometry group of $Y$ is of dimension greater
than three. Then, 
$Y$ is isometric to a metric Lie group $X$ which is
(up to rescaling) to one of the following spaces: $\R^3$, $\H^3$,
$\H^2\times \R$, ${\rm Nil}_3$, or to $\sl$ endowed with a rotationally
symmetric left invariant metric. The spaces $\H^3$ and $\H^2\times \R$
can be seen as non-unimodular metric Lie groups with $D$-invariant
$D=1$ and $D=0$, respectively. In addition, the underlying Riemannian manifold
associated to $\sl$ endowed with a rotationally symmetric left invariant
metric also admits a different metric Lie group structure,
isomorphic to the one of $\H^2\times \R$. Therefore, in all cases,
$X$ admits algebraic open book decompositions by Proposition~\ref{clasopenbook}. }
\end{remark}

\subsection{The left invariant Gauss map of surfaces in
metric Lie groups.}
\label{sec:gauss}

Let $f\colon  \Sigma\la X$ be an immersed oriented surface in a
metric Lie group $X$, and let $G\colon \Sigma\flecha \S^2$ denote its
left invariant Gauss map (Definition~\ref{defG}). Note that the
definition of $G$ depends on the left invariant metric chosen on
$X$; however, we next  explain that the property of such a Gauss
map being a diffeomorphism is independent of the left invariant metric,
i.e., it only depends on the Lie group structure of $X$.

To be more precise, suppose that $f\colon  \Sigma \la X$ is an oriented smooth surface,
and let $G(2,T_eX)$ be the set of oriented two-dimensional planes in
$T_eX$ passing through the origin of $T_eX$. Then, we can consider
two naturally defined Gauss maps
\begin{equation}
\label{eq:Gleft}
G_{\mbox{\small left}}\colon \Sigma \to G(2,T_eX),\quad
G_{\mbox{\small right}}\colon \Sigma \to G(2,T_eX),
\end{equation}
defined by left (resp. right) translating the respective oriented tangent spaces
of $\Sigma$ to $G(2,T_eX)$. Note that these Gauss maps are defined
without any reference to a left invariant metric on $X$. Now, once
we choose a left invariant metric, we can define $\S^2
\subset T_eX$ to be the set of tangent unit vectors to $X$ at $e$,
which is naturally diffeomorphic to $G(2,T_eX)$ just by taking for
each vector $v\in \S^2$ the oriented two-plane orthogonal to $v$
(using the orientation of $T_eX$ and the chosen left invariant metric).
This shows that the left
invariant Gauss map $G\colon  \Sigma \to \S^2$ introduced in Definition~\ref{defG}
with respect to this metric is a diffeomorphism 
if and only if $G_{{\rm left}}$ has the same
property, and therefore proves our claim.  

The left invariant Gauss map is useful to detect two-dimensional subgroups of metric Lie groups:

\begin{lemma}(\cite[Lemma 3.9]{mpe11})
Let $f\colon  \Sigma\la X$ be an immersed oriented surface in a
metric Lie group $X$, and let $G\colon \Sigma\flecha \S^2$ denote its
left invariant Gauss map. Then $G$ is constant if and only if $f(\Sigma)$
is contained in a left coset of a two-dimensional subgroup of $X$.
\end{lemma}

\section{Algebraic open book decompositions and
embeddedness of spheres.}
\label{sec:openbooks}

The main objective of this section is to prove Theorem~\ref{thm3.4},
which in particular implies the embeddedness result stated in
Theorem~\ref{main}. As a first step, we will
prove a \emph{transversality lemma}. For this purpose, we note that
a (simply connected) metric Lie group $X$ admits some
two-dimensional subgroup $\Sigma$ if and only if $X$ is not
isomorphic to ${\rm SU(2)}$, or equivalently, if and only if $X$
is  diffeomorphic to $\R^3$ (see Theorem~3.6 in~\cite{mpe11}).

\begin{lemma}[Transversality Lemma]
\label{lemma3.2}
Let $f\colon  S\looparrowright X$ be an immersed oriented sphere in $X$
whose left invariant Gauss map $G$ is a
diffeomorphism. Suppose that $X$ admits a two-dimensional subgroup $\Sigma $.
Then:
\begin{enumerate}[(1)]
\item The quotient space $X/\Sigma =\{ g\Sigma \, \mid \, g\in X\} $ consisting of the
left cosets of $\Sigma $ is diffeomorphic to $\R$.
\item The set of left cosets of $\Sigma $ which intersect
$f(S)$ can be parameterized by the interval $[0,1]$, i.e., $\{
g(t)\Sigma  \, \mid \, t\in [0,1] \} $ are these cosets.
\item Each of the left cosets $g(0)\Sigma $
and $g(1)\Sigma $ intersects $f(S)$ at a
single point.
\item  For every $t\in (0,1)$, $g(t)\Sigma $ intersects
$f(S)$ transversely in a connected, immersed closed curve, and the
preimage under $f$ of this curve is a simple closed curve in $S$.
\item Consider the smooth map
\[
\Pi^{\Sigma}_{\rm left}\colon  X \to X/\Sigma \cong\R, \quad \Pi^{\Sigma}_{\rm left}(g)=g\Sigma .
\]
 Then, $\Pi^{\Sigma}_{\rm left} \circ f$ is a Morse function with exactly two
critical points, where one critical point has index 0 and the other one  has index 2.
\end{enumerate}
\end{lemma}
\begin{proof}
Since $X$ admits a two-dimensional subgroup, then
$X$ is not isomorphic to $\su$ and so it is diffeomorphic to $\rth$.
In this case, the set of left cosets $X/\Sigma $
can be smoothly parameterized by $\R $. By compactness and connectedness of $f(S)$, we can
parameterize those left cosets that intersect $f(S)$ by $t\in [0,1]\mapsto g(t)\Sigma $.
After identifying $X/\Sigma $ with $\R $, we can view $\Pi^{\Sigma}_{\rm left}
\colon  X\to \R $ as the related smooth quotient map.
The critical points of $\Pi^{\Sigma}_{\rm left}\circ f$ are those points of
$S$ where the value of $G$ is
one of the two unit normal vectors to $\Sigma $. Since $G$ is bijective, then
$\Pi^{\Sigma}_{\rm left}\circ f$ has at most two critical points. On the
other hand, $\Pi^{\Sigma}_{\rm left}\circ f$ has at least two critical
points: a maximum and a  minimum. From here, the proof of
the first four statements in the lemma is elementary.

 It follows from items~(3) and (4) that
$\Pi^{\Sigma}_{\rm left} \circ f$ has exactly two critical points,
one where it takes on its minimum value
and one where it takes on its maximum value.
Suppose that  $\Pi^{\Sigma}_{\rm left} \circ f$ had a
degenerate critical point $p$, which is say the global minimum
of $\Pi^{\Sigma}_{\rm left} \circ f$. We next describe how to perturb $f$
to an immersion $f_t\colon S \looparrowright X$ for $t\geq 0$ small, so that $f_0=f$ and
$\Pi^{\Sigma}_{\rm left} \circ f_t$ has at least three critical points
when $t\neq0$. To do this, take a
diffeomorphism $\phi \colon X\to \rth$ satisfying:
\ben[A.]
\item $\phi (f(p))=(0,0,0)$ and $\phi (g\SS)=\R^2\times \{ \Pi^{\Sigma}_{\rm left}(g)\}$, for every $g\in X$.
\item Near $(0,0,0)$, $(\phi \circ f)(S)$ is expressed as the graph of a function $z=z(x,y)\geq0$
such that $z(0,0)=z_x(0,0)=z_y(0,0)=z_{xx}(0,0)=z_{xy}(0,0)=0$ and $z_{yy}(0,0)\geq 0$.
\een
Then, the family of diffeomorphisms $\{ Q_t\colon \R^3\to \R^3\ | \ t\in \R \} $ given by
$Q_t(x,y,z)=(x,y,z-t^2x^2)$ satisfies:
\ben[(a)]
\item $f_t=\phi ^{-1}\circ Q_t\circ \phi \circ f$ is a smooth 1-parameter of family
of immersions of $S$ into $X$ with $f_0=f$.
\item $\Pi^{\Sigma}_{\rm left}\circ f_t$ has a critical point
at $p$ with $(\Pi^{\Sigma}_{\rm left}\circ f_t)(p)=0$ (because $dQ_t(0,0,0)$ is the identity).
\item  For $t\neq0$ sufficiently small, near the point $(0,0,0)$, the surface $(\phi\circ f_t)(S)$ can
expressed as the graph of  a function $z_t=z_t(x,y)$, where  
$z_t(1/n,0)$ is negative for $n\in \N$ sufficiently large. 
Therefore, for $t>0$  sufficiently small, 
the global minimum value of $\Pi^{\Sigma}_{\rm left}\circ f_t$
is negative. 
\item  For $t>0$  sufficiently small,
the global maximum value of $\Pi^{\Sigma}_{\rm left}\circ f_t$ is close to 1. Therefore,
for $t>0$  sufficiently small,
$\Pi^{\Sigma}_{\rm left}\circ f_t$ must have  at least three critical points.
\een

By the openness of
the subspace of smooth immersed spheres in
$X$ whose left invariant Gauss maps are diffeomorphisms,
for $t>0$ sufficiently small,  the left invariant Gauss map
of $f_t  \colon S \looparrowright X$ is a
diffeomorphism.
But the existence of more than two critical points
of $\Pi^{\Sigma}_{\rm left} \circ f_t$ given in item (d) above
contradicts that items~(3) and (4)
 must hold for the perturbed immersion
$f_t$. Hence,
$\Pi^{\Sigma}_{\rm left} \circ f$ is a  Morse function with exactly two
critical points, where one critical
point (the minimum) has index 0 and the other one (the maximum)
has index 2. This
completes the proof of item~(5), and the lemma is proved.
\end{proof}

\begin{remark} {\em
Let $\Sigma $ be a two-dimensional subgroup and
consider the two oriented tangent planes $\pm T_e\SS\in G(2,T_eX)$  of $\SS$
at the identity element. Let $f\colon  S\looparrowright X$ be an immersed oriented sphere in $X$
 such that each of $T_e\Sigma $, $-T_e\Sigma $ has exactly one preimage
in $S$ by the map $G_{\mbox{\small left}}\colon S \to G(2,T_eX)$
introduced in (\ref{eq:Gleft}). Suppose that these preimages are regular
points of $G_{\mbox{\small left}}$. Then the conclusions of
Lemma~\ref{lemma3.2} hold with small
modifications in the proofs.}
\end{remark}

Next, we investigate some additional  properties of immersed spheres in
metric Lie groups $X$ that admit an algebraic open book decomposition.
Recall that, as explained in Subsection~\ref{sec:gauss},
the property of the Gauss map being a diffeomorphism only depends
on the Lie group structure of $X$ (and not
on the metric). Therefore, in the next two results
we do not lose generality by assuming that $X$ is equipped with the
canonical metric of $\R^2\rtimes_A \R$, where $A$ is
given in one of the Examples~~\ref{example1}, \ref{example2},
\ref{example3} or in equation~\eqref{Axieta}.

\begin{lemma}
\label{ass3.10}
Let $X$ be a metric Lie group which admits an algebraic open book
decomposition ${\mathcal B}$ with binding $\G$. If
$f\colon S\looparrowright X$ is an immersion of a sphere
whose left invariant Gauss map in $X$ is a diffeomorphism, then for
every $x\in X$, the set $f^{-1}(x\, \G )$ contains at most two points. Furthermore,
if $f^{-1}(x\, \G )$ consists of two points, then $x\G $ is transverse to $f$.
\end{lemma}

\begin{proof}
As explained in Subsection~\ref{subsec:openbook}, we can assume
that $X$ is a semidirect product $\R^2\rtimes_A\R$, that the
binding $\Gamma$ is the $1$-parameter subgroup $\{(x,0,0) \, \mid \, x\in \R\}$
and that the $(x,y)$-plane is one of the subgroups of $\cB$.

To prove that given $x\in X$ the set $f^{-1}(x\, \G )$
contains at most two points, we proceed by contradiction:
suppose that this assertion fails to
hold. We first find a contradiction in the case $f$ is analytic. In this case,
there is some $x\in X$ for which the compact analytic set
$f^{-1}(x\, \G )$ has more than two points; note that this set
has a finite number $k\geq 3$ of elements since $f(f^{-1}(x\, \G ))$ is a finite (compact)
analytic set of the non-compact curve $x\G$ and $f$
is a finite-to-one mapping. After left translating $f(S)$
by $x^{-1}$ (this does not change the left invariant Gauss map of
$f$), we can assume $\G =x\, \G $.

We claim that by choosing some $a\in
X$ arbitrarily close to the identity element $e$, the left translated binding $a\, \G$ is
transverse to $f(S)$ and the cardinality of $f^{-1}(a\, \G )$ is an
even integer greater than or equal to four. This fact can
be seen as follows. Since the $(x,y)$-plane  is a subgroup $H(\t_0)$
of $\mathcal{B}$ and $f^{-1}(H(\t _0))$ has more than one point, then
the Transversality Lemma~\ref{lemma3.2} implies
that $H(\t_0)$ intersects $f(S)$ transversely along a connected,
analytic immersed closed curve $\a$. Again by Lemma~\ref{lemma3.2},
$f^{-1}(\a)$ is a connected simple closed curve in $S$ which covers
the curve $\a$ (considered to be an immersion of $\esf^1$) with
fixed integer multiplicity $m\geq 1$. Suppose for
the moment that $m=1$. Since $\a$ is analytic, the
self-intersection set of $\a$ is finite. Since the $x$-axis $\G
\subset H(\t_0)$ intersects $\a$ in at least
three points (we are using that $m=1$ here), then,
elementary transversality theory ensures that for some
$t\neq 0$ sufficiently small and for $a(t)=(0,t,0)\subset H(\t_0)\subset \R^2\rtimes_A
\R $, $a(t)\, \G$ intersects $f(S)$ transversely in an even number of
points greater than 2, all of which are disjoint from the self-intersection set of
$\a$. Moreover, under these conditions $f^{-1}(a(t)\, \G )$
contains at least four points. This completes the proof of the claim if $m=1$. If $m>1$, then
for some $t\neq 0$ sufficiently small, $a(t)\, \G$
intersects $f(S)$ transversely in at least two points,
with  all of these intersection points being disjoint from the self-intersection set of
$\a$. Therefore,  $f^{-1}(a(t)\, \G)$ contains at
least $2m\geq 4$ points and our claim also holds in this case.

By the discussion in the previous paragraph and
after left translating $f(S)$
by $a^{-1}$, we may assume that $\G$
is transverse to $f(S)$ and $f^{-1}(\G )=\{ p_1,\ldots ,p_{2n}\}
\subset S$ with $n$ an integer, $n\geq 2$. Now consider the normal
variational vector field $\partial _{\t }$ to the leaves of the
product foliation $\mathcal{B}=\{ L(\t )\ | \ \t \in [0,2\pi )\} $
(these $L(\t)$ are the topological open halfplanes that appear in Definition~\ref{def:al-book}),
which is defined in $X-\G $. Since each leaf $L(\t )$ intersects $S$
transversely (again by the Transversality Lemma~\ref{lemma3.2}),
then the pullback by $f$ of the tangential component
of the restriction of $\partial _{\t }$ to $f(S)-\G $, defines a
vector field $\partial _{\t }^T$ on $S-f^{-1}(\G )= S-\{ p_1,\ldots
,p_{2n}\} $, and $\partial _{\t }^T$ has no zeros in $S-f^{-1}(\G
)$. The fact that the immersion $f$ is transverse to the binding $\G
$ of $\mathcal{B}$ implies that the index of $\partial _{\t }^T$ at
each of the points $p_j$ is $+1$. Then, by the Poincar\'e-Hopf index theorem,
the Euler characteristic of $S$ would be $2n\geq  4$, which is false
since the Euler characteristic of $S$ is 2. This is the desired contradiction in the
case that $f$ is analytic; thus, given $x\in X$, the
set $f^{-1}(x\, \G )$ contains at most two points in this case.

We next prove that  given $x\in X$ the set $f^{-1}(x\, \G )$
contains at most two points in the smooth case for $f$.
Arguing again by contradiction, suppose that for some $x\in X$, the set
$f^{-1}(x\, \G )$ has at least three points, $p_1,p_2,p_3$.
As before, we can assume $x\G =\G $.
By the Transversality Lemma~\ref{lemma3.2}, there exist
pairwise disjoint compact disks $D_1,D_2,D_3\subset S$
such that $p_i\in \Int (D_i)$ and $f(D_i)$ intersects
the $(x,y)$-plane in smooth compact arcs $\be _i$, $i=1,2,3$.
By perturbing $f$ slightly in $\Int(D_i)$, we can assume that
\ben[(A)]
\item Each $\be _i$ intersects $\G $ transversely at some point.
\item The left invariant Gauss map of the perturbed $f$ is a diffeomorphism.
\een
Since we can approximate
$f$ by analytic immersions with the properties (A) and (B), then we contradict the previously proved
analytic case for $f$. This contradiction finishes the proof of the first statement of the lemma.

It remains to show that if $f^{-1}(x\G)$ consists of two points $p_1,p_2$
for some $x\in X$, then $x\G$ is transverse to $f$.
If not, then $x\G $ intersects tangentially to $f(S)$ in at least one point say,
$f(p_1)$. After an arbitrarily  small smooth perturbation $\wh{f}$ of $f$ in a  small compact disk
neighborhood $D_1$ of  $p_1$ that is disjoint from $p_2$, we can suppose
that $\wh{f}|_{D_1}$ is injective and intersects $x\SS$ transversely in
an embedded arc, and this arc intersects
$x\G$  in at least two points. It follows that $\wh{f}^{-1}(x\G)\cap D_1$ contains at least 2 points and
$p_2 \in \wh{f}^{-1}(x\G)\cap (S-D_1)$, which means that $\wh{f}^{-1}(x\G)$ contains at least 3 points.
As in the previous paragraph, the
left invariant Gauss map of $\wh{f}$ can be assumed to be a diffeomorphism,
which contradicts the previously proved statement that the set
$\wh{f}^{-1}(x\G)$
contains at most 2 points. This contradiction completes the proof that
if $f^{-1}(x\G)$ consists of two points $p_1,p_2$
for some $x\in X$, then $x\G$ is transverse to $f$.
\end{proof}

\begin{theorem}
 \label{thm3.4}
Let $X$ be a metric Lie group which admits an algebraic open book
decomposition ${\mathcal B}$ with binding $\G$. Let $\Pi\colon  X\to
X/\G \cong \R^2$ be the related quotient map to the space of left
cosets of $\G$. If $f\colon  S\looparrowright X$ is an oriented immersion of a
sphere whose
left invariant Gauss map 
is a diffeomorphism, then:
\begin{enumerate}[(1)]
\item ${\mathcal D}=\Pi(f(S))$ is a 
compact embedded disk in $\R^2$ and $f(S)\cap\Pi^{-1}({\rm Int}({\mathcal D}))$
consists of two components $C_1, C_2$ such that $\Pi|_{C_j} \colon  C_j \to
{\rm Int}({\mathcal D})$ is a diffeomorphism  for $j=1,2$.
\item $f(S)$ is an embedded sphere (i.e., $f$ is an injective immersion).
\end{enumerate}
\end{theorem}

\begin{proof}
As explained in Section~\ref{subsec:openbook},
we can
assume without loss of generality that $X$ is a semidirect product
$\R^2\rtimes _A\R $ equipped with its canonical metric, the binding
$\G $ is $\{ (x,0,0)\ | \ x\in \R \}$ and $P_0=\R^2\rtimes _A\{ 0\} $
is one of the subgroups of $\mathcal{B}$.  It also follows from the
classification in~\cite[Theorem~3.6]{mpe11}
of the algebraic open book decompositions
in metric Lie groups, that the $2\times 2$-matrix $A$ can be chosen so that
its position $a_{21}$ vanishes. This implies that given $z\in \R $,
the matrix $e^{zA}$ is upper triangular, which gives that
each left translate of $\G$ that intersects $P_z=\R^2\rtimes _A\{ z\}$
corresponds to a straight line in $P_z$ that is of the form
$\G_{y,z}=\{(x,y,z) \mid x\in \R\}$.
Let $z_0<z_1$ be the numbers such that $f(S)$ is contained in the
region $\R^2\rtimes _A [z_0,z_1]$ and intersects each of the planes
$P_{z_0}$, $P_{z_1}$ at single points $p_0, p_1$, respectively (see
Lemma~\ref{lemma3.2}). After a left translation, we can suppose $z_0=0$.

Orient the binding $\G=\{(x,0,0) \mid x\in \R\}$ by the usual orientation on $\R$.
Also orient the sphere $S$, orient the planes $P_z$ by
$E_3=\partial_z$ and the space $\R^2\rtimes _A\R $ by the ordered
triple $E_1,E_2,E_3$ given in (\ref{eq:6*}). By
Lemma~\ref{lemma3.2}, each of the oriented planes $P_z$, $z\in
(0,z_1)$, is transverse to $f$ and so $f^{-1}(P_z)$ is a smooth,
embedded, oriented Jordan curve $\a_z$ in $S$, where the orientation
is the homological one arising from the ordered intersection
$P_z\cap f(S)$. Note that with respect to the induced metric, each
plane $P_z$ is intrinsically flat and it is foliated by the collection
of parallel lines $\G_{y,z}=\{(x,y,z) \mid x\in \R\}$ as $y\in \R$ varies.
Thus, we can view the quotient map $\Pi\colon  X\to X/\G \cong \R^2$ as the
natural projection $(x,y,z)\in \R^2\rtimes _A\R \mapsto (y,z) \in \R
^2$ in the standard coordinates of $\R^2\rtimes _A\R $.

By
Lemma~\ref{lemma3.2}, for each $z\in (0,z_1)$, $\Pi(f(\a_z))=I_z$
is a compact interval with nonempty interior, $I_z$
is contained in the line $\{(0,y,z)\ | \ y\in \R \}$ and
the union of these intervals $I_z$ for $z\in (0,z_1)$ together
with the two points $\Pi (p_0), \Pi (p_1)$ is the compact embedded
disk ${\mathcal D}=\Pi(f(S))$. By Lemma~\ref{ass3.10}, the interior of each of the intervals
$I_z$ lifts through the local diffeomorphism $\Pi $ to two open connected arcs $I_z^1,
I_z^2\subset f(\a_z)$ with common extrema, where the superindices
$j=1,2$ are consistently defined for all $z$, so that the velocity
vector to $I_z^1$ at its starting extremum (resp. ending extremum)
points at the direction of $\partial _x$ (resp. of $-\partial _x$), and the velocity
vector to $I_z^j$ at any of its (interior) points is never in the direction of $\pm \partial _x$,
see Figure~\ref{figthm3-10}.
\begin{figure}
\begin{center}
\includegraphics[height=6.5cm]{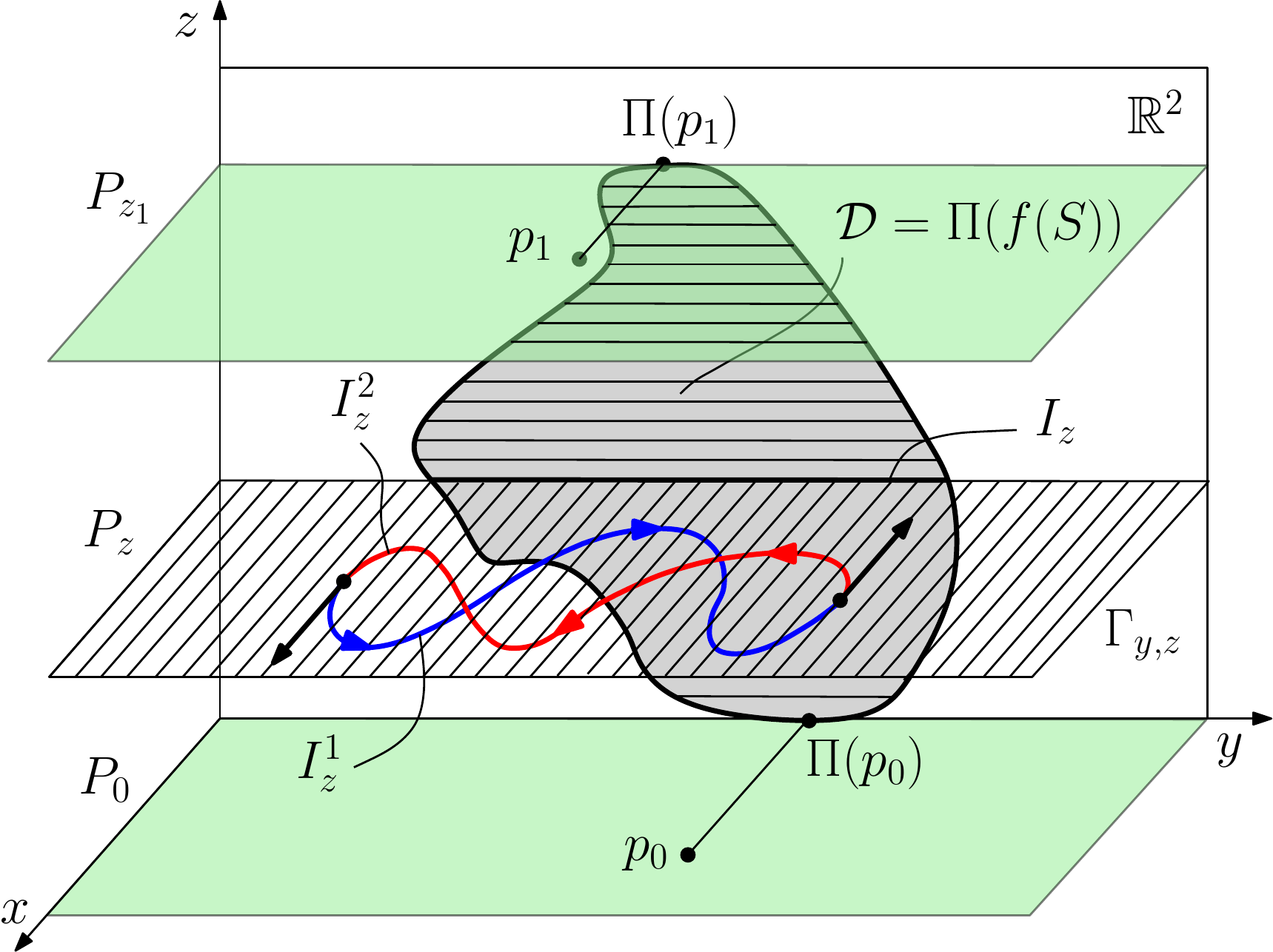}
\caption{Proof of Theorem~\ref{thm3.4}; the immersed curve $f(\a
_z)=f(S)\cap P_z$ consists of two Jordan arcs $I_z^1,I_z^2$ with the
same extrema.} \label{figthm3-10}
\end{center}
\end{figure}
For $j=1,2$ fixed, let $C_j$ be the open disk in
$f(S)$ given by the union of the arcs $I_z^j$ as $z$ varies.
Thus, we have proved that $f(S)\cap\Pi ^{-1}(\mbox{Int}({\mathcal
D}))$ decomposes as a disjoint union of the disks $C_1,C_2$ with
common boundary $\partial C_1=
\partial C_2$, each $C_j$ is $\Pi $-graphical
onto $\mbox{Int}({\mathcal D})$, and $\Pi |_{C_j}$ is a
diffeomorphism from $C_j$ onto $\Int({\mathcal D})$.
This completes the proof of item~(1) of the theorem.

To prove item~(2), 
first assume that
$X=\R^2\rtimes _A\R $ for some diagonal matrix $A\in
\mathcal{M}_2(\R )$. Then, as shown in Section~3 of~\cite{mpe11}
(see in particular~\cite[Example 3.4]{mpe11}), $X$ admits two
different algebraic open book decompositions $\cB_1,\cB_2$ with
respective orthogonal bindings $\G_1=\{ (x,0,0)\ | \ x\in \R \} $,
$\G_2=\{ (0,y,0)\ | \ y\in \R \} $. By our previous arguments, for
$z\in (0,z_1)$, the immersed curve $\a_{z}$ can be also expressed
as a ``bigraph" over appropriately chosen intervals in the lines
$\{(0,y,z)\mid y\in \R\}$, $\{(x,0,z)\mid x\in \R\}$, which implies
that each such $\a_z$ is embedded.  Since all of the curves $\a_z$
are embedded, then $f(S)$ is an embedded sphere. This proves that
item~(2) holds provided that $X=\R^2\rtimes _A\R $ for some diagonal
matrix $A\in \mathcal{M}_2(\R )$.

Finally assume that $X$ admits an algebraic open book decomposition,
but is not isomorphic to $\R^2\rtimes _A\R $ for some diagonal
matrix $A\in \mathcal{M}_2(\R )$.  By Theorem~3.6 in~\cite{mpe11},
$X$ is isomorphic either to Nil$_3$ or to the simply connected,
three-dimensional  Lie group with $D$-invariant $D=1$ which is not
isomorphic to $\Hip ^3$. First consider the case where the underlying
Lie group $G$ is the one with $D$-invariant $D=1$ which is not
isomorphic to $\Hip ^3$, and fix the related matrix to be
$A=\left(\begin{array}{cr} 1 & 1\\ 0& 1\end{array}\right)$. Then $X=\R^2\rtimes_A\R$
is isomorphic to~$G$.  Given $n\in \N$,
let $A(n)=\left(\begin{array}{cr} 1 & 1\\ 1/n^2 & 1\end{array}\right)$
and let $X_n=\R^2\rtimes_{A(n)}\R$ be the associated
metric Lie group.  Since $\lim_{n\to \infty} A(n)=A$,
then in the associated $\rth=\R^2\times \R$-coordinates
on these spaces $X_n$, compact balls in these coordinates
for $X_n$ converge smoothly as Riemannian manifolds to the
related compact balls in these coordinates for $X$.  Since $X_n$ is isomorphic  to a non-unimodular Lie group
with $D$-invariant $D_n=1- 1/n^2$, then $X_n$ is also isomorphic
to the Lie group $\R^2\rtimes _{B_n}\R $
where $B_n$ is the diagonal matrix
\[
B_n=\left(\begin{array}{cr} 1+1/n & 0\\ 0& 1-1/n\end{array}\right).
\]
Note that we can view $f\colon S\looparrowright X$ also as an immersed sphere $f\colon S\looparrowright X_n$.
Furthermore, as explained just before
the statement of Lemma~\ref{ass3.10}, the property that
the left invariant Gauss map of
$f$ is a diffeomorphism is independent of the left invariant
metric chosen on $X$ (or on $X_n$).
In particular, by the convergence of $X_n$ to $X$ explained above, for $n$ large enough, the left
invariant Gauss map of the immersed sphere $f(S)$ in $X_n$ is a
diffeomorphism (for any left invariant metric on $X_n$).
Since $B_n$ is a diagonal matrix, then $f\colon S\looparrowright X_n$
is an embedding by the previous paragraph. Finally, suppose
that $f\colon  S\la X$ is not an embedding. Then, a small perturbation $f_{\varepsilon}$
of $f\colon  S\la X$ in the same ambient space can be assumed
to intersect itself transversely at some point, and to still have the property that its left invariant Gauss map is a diffeomorphism.
This clearly implies by the previous arguments
that $f_{\varepsilon}\colon S\looparrowright X_n$ is not an embedding but its left invariant Gauss map is a diffeomorphism. This is a contradiction that completes the proof of the theorem when the Lie group
structure of $X$ is the one with $D$-invariant $D=1$ which is not
isomorphic to $\Hip ^3$.

Finally, if $X$ is isomorphic to Nil$_3$, then one can do a
similar argument as above, taking into account that
for every non-unimodular Lie group different from $\Hip ^3$ there
exists a sequence of left invariant metrics on it such that the corresponding
sequence of metric Lie groups converges to Nil$_3$ with its
standard metric, see the last paragraph in Section~2.8 of~\cite{mpe11}
for details on this argument.
Now the proof is complete.
\end{proof}

\begin{remark}\label{torralbo}

\emph{Let $X$ be a homogeneous manifold diffeomorphic to $\S^3$.
Then, we can view $X$ as the Lie group $\su$ endowed with a left
invariant metric. We show next that there exist immersed spheres in
$X$ whose left invariant Gauss map is a diffeomorphism to $\S^2$, but
that are not embedded, i.e., they self-intersect. Note that this
statement is independent of the left invariant metric chosen
on $\su$ for $X$; see Subsection~\ref{sec:gauss}.}

\emph{A metric Lie group $X=(\su, g_0)$ where $g_0$ is a left invariant
metric on $\su$ with a four-dimensional isometry group is usually called
a \emph{Berger sphere}. In~\cite{tor1}, Torralbo showed that in some
Berger spheres there exist rotationally symmetric constant mean curvature
spheres that are non-embedded. On the other hand, the authors proved
with Ros in~\cite{mmpr4} that if $S$ is a constant mean curvature sphere
in a homogeneous manifold diffeomorphic to $\S^3$, then the left invariant
Gauss map of $S$ is a diffeomorphism to $\S^2$.}

\emph{These two facts together prove the claimed existence of non-embedded
spheres in $\su$ whose left invariant Gauss map is a diffeomorphism.}
\end{remark}

\begin{remark} \label{rem:left-right} {\em
Let $\star $ denote the group multiplication of
a metric Lie group $X$ that admits an algebraic open book decomposition and let
$\circ $ be the ``opposite" multiplication:
$a\circ b=b\star a$, \ $a,b\in X$. Then the map $\sigma(a)= a^{-1}$ gives
a Lie group isomorphism between $(X,\star )$ and $(X,\circ )$.
Let $f\colon  S\looparrowright X$ be an immersed oriented sphere in $(X,\star)$.
Then, using the notation in (\ref{eq:Gleft}), the right invariant Gauss map
$G_{\mbox{\small right}}$ of $f$ in $(X,\star )$ is the
left invariant Gauss map $\wh{G}_{\mbox{\small left}}$ of the immersion
$f\colon S\looparrowright (X,\circ )$.
Since the algebraic open
book decompositions of $(X,\circ )$ are the same as the algebraic open
book decompositions of $(X,\star )$,
then Theorem~\ref{thm3.4} implies that
if $f\colon  S\looparrowright X$ is an
immersion of a sphere whose right invariant Gauss map is a
diffeomorphism, then $f(S)$ is an embedded
sphere.
}
\end{remark}

\begin{remark}{\em
In view of Theorem~\ref{main} and the
previous remark, we have  the following
natural open problem: \emph{Let $X$ be a Lie group
diffeomorphic to $\R^3$, and let $S$ be an immersed sphere in $X$ whose
left (or right) invariant Gauss map  is a
diffeomorphism to $G(2,T_eX)$. Is $S$ an embedded sphere?}.

By work in~\cite{mmpr4}, this problem
is closely related to the conjecture mentioned in the introduction: {\em Any constant
mean curvature sphere in a homogeneous
manifold diffeomorphic to $\R^3$ is embedded.}}
\end{remark}

\bibliographystyle{plain}
\bibliography{bill}

\begin{thebibliography}{1}

\bibitem{abro1}
U.~Abresch and H.~Rosenberg.
\newblock A {H}opf differential for constant mean curvature surfaces in
  ${\esf}^2\times\mathbb{R}$ and $\mathbb{H}^2\times \mathbb{R}$.
\newblock {\em Acta Math.}, 193(2):141--174, 2004.
\newblock MR2134864 (2006h:53003), Zbl 1078.53053.

\bibitem{abro2}
U.~Abresch and H.~Rosenberg.
\newblock Generalized {H}opf differentials.
\newblock {\em Mat. Contemp.}, 28:1--28, 2005.
\newblock MR2195187, Zbl 1118.53036.

\bibitem{dm2}
B.~Daniel and P.~Mira.
\newblock Existence and uniqueness of constant mean curvature spheres in
  {S}ol$_3$.
\newblock {\em Crelle: J Reine Angew Math.}, 685:1--32, 2013.
\newblock MR3181562, Zbl 1305.53062.

\bibitem{mmpr4}
W.~H. Meeks~III, P.~Mira, J.~P\'{e}rez, and A.~Ros.
\newblock Constant mean curvature spheres in homogeneous three-spheres.
\newblock Preprint at http://arxiv.org/abs/1308.2612.

\bibitem{mpe11}
W.~H. Meeks~III and J.~P\'{e}rez.
\newblock Constant mean curvature surfaces in metric {L}ie groups.
\newblock In {\em Geometric Analysis}, volume 570, pages 25--110. Contemporary
  Mathematics, edited by J. Galvez, J. P\'{e}rez, 2012.
\newblock MR2963596, Zbl 1267.53006.

\bibitem{mil2}
J.~W. Milnor.
\newblock Curvatures of left invariant metrics on {L}ie groups.
\newblock {\em Advances in Mathematics}, 21:293--329, 1976.
\newblock MR0425012, Zbl 0341.53030.

\bibitem{so3}
R.~Souam.
\newblock On stable constant mean curvature surfaces in {${\mathbb S}^2 \times
  \R$} and {${\mathbb{H}}^2\times \R$}.
\newblock {\em Trans. Am. Math. Soc.}, 362(6):2845--2857, 2010.
\newblock MR2592938, Zbl 1195.53089.

\bibitem{tor1}
F.~Torralbo.
\newblock Rotationally invariant constant mean curvature surfaces in
  homogeneous $3$-manifolds.
\newblock {\em Differential Geometry and its Applications}, 28:523--607, 2010.
\newblock MR2670089, Zbl 1196.53040.

\bibitem{tou1}
F.~Torralbo and F.~Urbano.
\newblock Compact stable constant mean curvature surfaces in homogeneous
  $3$-manifolds.
\newblock {\em Indiana Univ. Math. Journal}, 61:1129--1156, 2012.
\newblock MR3071695, Zbl 1278.53065.

\end{thebibliography}
\end{document}